\documentclass[12pt]{smfart}
\usepackage{amscd,verbatim}
\usepackage[french]{babel}
\usepackage[T1]{fontenc}
\usepackage[colorlinks,linkcolor=blue,citecolor=blue,urlcolor=red]{hyperref}

\newcommand{\C}{\mathbf{C}}
\newcommand{\F}{\mathbf{F}}
\renewcommand{\P}{\mathbf{P}}
\newcommand{\Q}{\mathbf{Q}}
\newcommand{\sX}{\mathcal{X}}
\newcommand{\sY}{\mathcal{Y}}
\newcommand{\Coker}{\operatorname{Coker}}
\newcommand{\IM}{\operatorname{Im}}
\newcommand{\Pic}{\operatorname{Pic}}
\newcommand{\NS}{\operatorname{NS}}
\newcommand{\cl}{\operatorname{cl}}
\newcommand{\tr}{{\operatorname{tr}}}
\renewcommand{\hom}{{\operatorname{hom}}}
\newcommand{\car}{\operatorname{car}}

\newcommand{\by}{\xrightarrow}
\newcommand{\iso}{\by{\sim}}

\newcommand{\inj}{\hookrightarrow}

\DeclareFontFamily{U}{wncy}{}
\DeclareFontShape{U}{wncy}{m}{n}{%
<5>wncyr5%
<6>wncyr6%
<7>wncyr7%
<8>wncyr8%
<9>wncyr9%
<10>wncyr10%
<11>wncyr10%
<12>wncyr6%
<14>wncyr7%
<17>wncyr8%
<20>wncyr10%
<25>wncyr10}{}
\DeclareMathAlphabet{\cyr}{U}{wncy}{m}{n}
\newcommand{\Sha}{\cyr{X}}

\newtheorem{thm}{Théorème}
\newtheorem{lemme}{Lemme}
\newtheorem{prop}{Proposition}

\theoremstyle{remark}
\newtheorem{rque}{Remarque}
\newtheorem{qn}{Question}

\begin{document}
\title{Sur la conjecture de Tate pour les diviseurs}
\author{Bruno Kahn}
\address{IMJ-PRG\\Case 247\\
4 place Jussieu\\
75252 Paris Cedex 05\\France}
\email{bruno.kahn@imj-prg.fr}
\date{$1^{er}$ février 2023}
\begin{altabstract} We prove that the Tate conjecture in codimension $1$ over a finitely generated field follows from the same conjecture for surfaces over its prime subfield. In positive characteristic, this is due to de Jong--Morrow over $\F_p$ and to Ambrosi for the reduction to $\F_p$. We give a different proof than Ambrosi's, which also works in characteristic $0$; over $\Q$, the reduction to surfaces follows from a simple argument using Lefschetz's $(1,1)$ theorem.
\end{altabstract}
\begin{abstract} On montre que la conjecture de Tate en codimension $1$ sur un corps de type fini résulte de la même conjecture pour les surfaces sur son sous-corps premier. En caractéristique positive, ceci est dû à de Jong--Morrow sur $\F_p$ et à Ambrosi  pour la réduction à $\F_p$. Nous montrons cette dernière réduction d'une manière différente, qui fonctionne aussi en caractéristique zéro. Sur $\Q$, la réduction aux surfaces se fait par un argument facile reposant sur le théorème $(1,1)$ de Lefschetz.
\end{abstract}
\maketitle

\section*{Introduction} La conjecture de Tate est l'une des plus célèbres en géométrie arithmétique: formulée en 1965 dans \cite{tate0}, elle prédit que l'application classe de cycle $l$-adique \cite{cycle}
\begin{equation}\label{eq3}
\cl^i_X:CH^i(X)\otimes \Q_l\to H^{2i}(X_{k_s},\Q_l(i))^{Gal(k_s/k)}
\end{equation}
est surjective pour tout entier $i\ge 0$ et toute variété projective lisse $X$ sur un corps $k$ de type fini, de caractéristique différente de $l$ et de clôture séparable $k_s$. Elle a été démontrée dans de nombreux cas particuliers, mais reste ouverte en général même pour $i=1$. Pour un exposé détaillé qui reste largement d'actualité, je renvoie à \cite{tate} (voir aussi \cite{li-zhang}).

Il est connu que pour $i=1$, la conjecture de Tate pour les corps premiers l'implique en général: en caractéristique zéro cela se déduit du théorème de spécialisation des groupes de Néron-Severi dû à Yves André (\cite[Theorem 5.2 3]{andre}, \cite[1.3.3]{ambrosi}), et en caractéristique positive cela résulte d'un théorème d'Emiliano Ambrosi \cite[th. 1.2.1]{ambrosi}. Ambrosi démontre plus:  la conjecture de Tate en codimension $i$ sur les corps finis l'implique pour tous les corps de type fini de caractéristique positive, sous une hypothèse de semi-simplicité qui résulte de la conjecture de Tate quand $i=1$. Sa preuve,  étendant au cas d'un corps fini un argument d'André en caractéristique zéro \cite[\S 5.1]{andre}, utilise le théorème global des cycles invariants de Deligne et la cyclicité du groupe de Galois absolu de $\F_p$. 

L'objet de cette note est d'offrir une démonstration plus élémentaire de cette réduction (uniquement pour $i=1$), qui fonctionne uniformément en toute caractéristique: inspirée de la preuve de \cite[th. 8.32 a)]{sheaf}, elle consiste à étendre la conjecture de Tate aux variétés lisses \emph{ouvertes} (théorème \ref{t1}). Cette idée, due originellement à Jannsen \cite{jannsen}, permet de remplacer avantageusement le recours au théorème global des cycles invariants par une simple utilisation du critère de dégénérescence des suites spectrales de Deligne \cite{deligne1,deligne2}. Un argument élémentaire de correspondances permet par ailleurs de se débarrasser aisément du problème de semi-simplicité qui apparaît aussi chez Jannsen (lemme \ref{l1} et proposition \ref{p1}).  

D'après de Jong et Morrow \cite{morrow}, la conjecture de Tate pour $i=1$ en caractéristique positive se réduit même au cas des surfaces sur un corps fini (cette dernière conjecture étant par ailleurs équivalente à la conjecture de Birch et Swinnerton-Dyer pour les variétés abéliennes sur les corps de fonctions d'une variable sur $k$, voir remarque \ref{r2}). La même chose est vraie en caractéristique zéro (théorème \ref{t0}), en utilisant le théorème (1,1) de Lefschetz via les théorèmes de comparaison cohomologiques.

\newpage

\section*{English introduction} The famous Tate conjecture, which predicts that the $l$-adic cycle class map \eqref{eq3} is surjective for smooth projective varieties $X$ over finitely generated fields $k$, was formulated back in 1965 \cite{tate0}, but remains open up to now, even though it has been proven in important special cases \cite{tate,li-zhang}. This is so even for $i=1$.

In this case, the Tate conjecture over prime fields implies it in general by work of Emiliano Ambrosi \cite[th. 1.2.1 and 1.3.3]{ambrosi}. In characteristic $0$, the argument uses Yves André's specialisation theorem for the Néron-Severi group \cite[Theorem 5.2 3]{andre}, while in positive characteristic, Ambrosi's proof relies in particular on the cyclicity of $Gal(\bar \F_p/\F_p)$.

The aim of this note is to offer a simple proof of this reduction, which works uniformly in all characteristics (Proposition \ref{t2}). It is however special to codimension $1$, while Ambrosi's argument also works (in positive characteristic) in any codimension $i$ under a semi-simplicity hypothesis which follows from Tate's conjecture if $i=1$. 

There are two ideas: the first is to get rid of the semi-simplicity issue by a simple argument of correspondences (Lemma \ref{l1}), and the second is to extend the Tate conjecture for divisors from smooth projective to (all) smooth varieties (Theorem \ref{t1}). This idea goes back to Jannsen \cite{jannsen}; its point is that it allows us to replace Ambrosi's use of Deligne's global invariant cycles theorem (an argument going back to André in characteristic $0$ \cite[\S 5.1]{andre}) by the degeneracy of the $l$-adic Leray spectral sequence, also due to Deligne. This is a reformulation of the arguments given in  \cite[proof of th. 8.32 a)]{sheaf}, specialised to codimension $1$ (see also \cite[Th. 3.4]{glr}); the first idea is new.

When $k=\F_p$, a theorem of de Jong and Morrow \cite{morrow} even reduces the Tate conjecture for divisors to surfaces. (This case is in turn equivalent to the Birch and Swinnerton-Dyer conjecture for abelian varieties over global fields of positive characteristic, see Remark \ref{r2}.) Over $\Q$, the same reduction holds (Theorem \ref{t0}): the proof involves the Lefschetz (1,1) theorem via the cohomological comparison theorems. 

\enlargethispage*{30pt}

\section{Notations} Soient $k$ un corps et $X$ une $k$-variété lisse. Soient $k_s$ une clôture séparable de $k$ et $l$ un nombre premier différent de $\car k$. On note $H^j(X,i):=H^j(X_{k_s},\Q_l(i))$; de même pour la cohomologie à supports. On note $\cl^i_X:CH^i(X)\otimes \Q_l\to H^{2i}(X,i)$ la classe de cycle, et simplement $\cl_X$ pour $\cl^1_X$. 

\section{Une rétraction}\label{s1}

Supposons $X$ projective de dimension $d$. Pour $i\le d$, choisissons une base $(\bar Z^1,\dots,\bar Z^r)$ du groupe $N^i(X)_\Q$ des cycles de codimension $i$ sur $X$ modulo l'équivalence numérique, à coefficients rationnels et notons $(\bar Z_1,\dots,\bar Z_r)$ la base duale dans $N_i(X)_\Q$, de sorte que $\langle\bar Z^i,\bar Z_j\rangle=\delta_{ij}$\footnote{Rappelons que $N^i(X)_\Q$ et $N_i(X)_\Q$ sont des $\Q$-espaces vectoriels de dimension finie \cite[ex. 19.1.4]{fulton}, mis en dualité par l'accouplement d'intersection.}. Relevons les $\bar Z^i$ and $\bar Z_i$ en des classes de  cycle $Z^i\in CH^i(X)_\Q$, $Z_i\in CH_i(X)_\Q$. Soit (\emph{cf.} \cite[dém. de la prop. 7.2.3]{kmp})
\[e=\sum_a Z^a\times Z_a\in CH^d(X\times X)_\Q\]
vu comme correspondance algébrique, où $\times$ est le cross-produit des cycles. 

\begin{lemme}\label{l0}  On a $e^2=e$.
\end{lemme}

\begin{proof} Pour $Z,Z'\in CH^i(X)$ et $T,T'\in CH_i(X)$, on a l'identité
\[(Z\times T)\circ (Z'\times T')=\langle Z',T\rangle Z\times T'\]
dans l'anneau des correspondances de Chow $CH^d(X\times X)$, où $\langle,\rangle$ est le produit d'intersection: cela résulte immédiatement de la définition de la composition des correspondances \cite[Déf. 16.1.11]{fulton}.
\end{proof}

\begin{lemme}\label{l1} Soit $V$ le sous-espace vectoriel de $\IM \cl^i_X$ engendré par les $Z^a$. L'action de $e$ sur $H^{2i}(X,i)$ définit une rétraction $G$-équivariante de l'inclusion $V\inj H^{2i}(X,i)$. En particulier, si $i=1$, elle définit une rétraction de $\cl_X$.
\end{lemme}

\begin{proof} Soit $x\in H^{2i}(X,i)$. Pour $(Z,T)\in CH^i(X)\times CH_i(X)$, on a
\[(Z\times T)^*x = < x,\cl_i(T)> \cl^i(Z)\]
où $<,>$ est l'accouplement de Poincaré, cf. \cite[déf. 16.1.2]{fulton}. Cela montre que $e(H^{2i}(X,i))\allowbreak\subset V$, et aussi que sa restriction à ce sous-espace est l'identité. 

Le cas $i=1$ résulte du théorème de Matsusaka \cite{matsusaka} (équivalences homologique et numérique coïncident en codimension $1$).
\end{proof}

\section{Passage aux variétés lisses ouvertes}\label{s2}

Supposons $k$ de type fini, et $X$ seulement lisse. On s'intéresse à l'extension suivante de la conjecture de Tate:

\begin{description} 
\item[$T(X)$] l'homomorphisme ``classe de diviseur'' $\cl_X:\Pic(X)\otimes \Q_l \to H^2(X,1)^{G}$ est surjectif, où $G=Gal(k_s/k)$.
\item[$T(k)$] $T(X)$ pour toutes les $k$-variétés lisses $X$.
\end{description}

Aux notations près, cette conjecture est due à Jannsen \cite[conj. 7.3]{jannsen}, qui l'étend même aux variétés singulières (avec l'homologie de Borel-Moore). Dans \cite[th. 7.10 b)]{jannsen}, il la réduit au cas des variétés projectives lisses sous une hypothèse de semi-simplicité (b), en bas de \cite[p. 113]{jannsen}) qui n'est pas connue en général même pour $H^2$.  Le but de cette section est de faire cette réduction (théorème \ref{t1}) en évitant l'hypothèse de semi-simplicité grâce à la proposition \ref{p1} ci-dessous. Bien sûr, ceci ne marche que pour les cycles de codimension $1$! 

Dans la suite, on note
\[H^2_\tr(X,1)=\Coker \cl_X.\]

\enlargethispage*{20pt}

\begin{prop}\label{p1} Pour $X$ projective lisse, $T(X)$ équivaut à $H^2_\tr(X,1)^G\allowbreak=0$.
\end{prop}

\begin{proof} L'implication $H^2_\tr(X,1)^G=0$ $\Rightarrow$ $T(X)$ est évidente. L'autre résulte du lemme \ref{l1}. \end{proof}

\begin{lemme}\label{l2} Soit $f:X'\to X$ un morphisme fini et plat de $k$-variétés lisses. Alors $T(X')\Rightarrow T(X)$.
\end{lemme}

\begin{proof} En effet, $f^*:H^2(X,1)\to H^2(X',1)$ admet la rétraction $G$-équivariante $(1/\deg(f))f_*$, et ces deux homomorphismes commutent avec les homomorphismes correspondants entre $\Pic(X)\otimes \Q_l$ et $\Pic(X')\otimes \Q_l$ via 
$\cl_X$ et $\cl_{X'}$.
\end{proof}

\begin{prop}\label{p2} a) Soit $U$ un ouvert de $X$. Alors $T(U)\Rightarrow T(X)$.\\
b) La réciproque est vraie si $X$ est projective.
\end{prop}

\begin{proof} 
Notons que $\cl_X$ se factorise par $\NS(X)\otimes \Q_l$ où $\NS(X)$ est le groupe de Néron-Severi de $X$. Soit $Z=X-U$ (structure réduite). On a un diagramme commutatif aux lignes exactes
\begin{equation}\label{eq1}
\tiny{\begin{CD}
\bigoplus\limits_{x\in Z\cap X^{(1)}} \Q_l@>>> \NS(X)\otimes \Q_l@>>> \NS(U)\otimes \Q_l@>>> 0\\
@V V V @V\bar\cl_X VV @V\bar \cl_U VV\\ 
H^2_Z(X,1)@>\delta >> H^2(X,1)@>>> H^2(U,1)@>>> H^3_Z(X,1)
\end{CD}}
\end{equation}
où la ligne du bas est la suite exacte de cohomologie à supports et la flèche verticale de gauche est surjective (en fait bijective) par semi-pureté \cite[2.2.6 et 2.2.8]{cycle}. 
On en déduit un nouveau diagramme commutatif aux lignes exactes
\begin{equation}\label{eq2}
\tiny{\begin{CD}
&&\bigoplus\limits_{x\in Z\cap X^{(1)}} \Q_l@>>> \NS(X)\otimes \Q_l@>>> \NS(U)\otimes \Q_l@>>> 0\\
&&@V V V @V\bar\cl_X VV @V\bar \cl_U VV\\ 
0@>>> \IM \delta@> >> H^2(X,1)^G@>>> H^2(U,1)^G
\end{CD}}
\end{equation}
où la flèche verticale de gauche est surjective. L'assertion résulte alors d'une petite chasse aux diagrammes.

b) Supposons d'abord $k$ parfait. D'après a), on peut choisir $U$ aussi petit qu'on veut.  Prenons $Z= X-U$ assez gros pour contenir des diviseurs $D_1,\dots, D_r$ dont les classes engendrent $\NS(X)$. Alors $\NS(U)=0$ et il faut montrer que $H^2(U,1)^G=0$. Or le diagramme \eqref{eq1} montre que la suite
\[\NS(X)\otimes \Q_l\by{\bar\cl_X} H^2(X,1)\to H^2(U,1)\to H^3_Z(X,1)\]
est exacte. Avec la notation de la proposition \ref{p1}, on a donc une suite exacte
\[0\to H^2_\tr(X,1)^G\to H^2(U,1)^G\to H^3_Z(X,1)^G.\]

Si $T(X)$ est vrai, le terme de gauche est nul par cette proposition. Il reste à voir que le terme de droite l'est également. Soit $Z'$ la réunion du lieu singulier de $Z$ et de ses composantes irréductibles de codimension $\ge 2$ dans $X$: la suite exacte de cohomologie à supports
\[0\simeq H^3_{Z'}(X,1)\to H^3_Z(X,1)\to H^3_{Z-Z'}(X-Z',1)\simeq H^1(Z-Z',0)\]
où le premier (\emph{resp.} second) isomorphisme est par semi-pureté (\emph{resp.} par pureté) \cite[2.2.8]{cycle}, montre que $H^3_Z(X,1)^G$ s'injecte dans $H^1(Z-Z',0)^G$. Mais ce dernier groupe est trivial, car $H^1(Z-Z',0)$ est mixte de poids $\ge 1$ \cite[cor. 3.3.5]{weilII}.\footnote{Au moins sur un  corps fini, ce dernier point peut se déduire plus élémentairement du théorème antérieur de A. Weil pour les courbes \cite[n° 48]{weil}, en utilisant le fait que $H^1(Z-Z',0)$ est isomorphe au module de Tate rationnel de la variété d'Albanese de $Z-Z'$ via un morphisme d'Albanese.} 

L'argument ci-dessus utilise implicitement le fait que les composantes irréductibles de codimension $1$ de $Z$ sont génériquement lisses. Pour obtenir ceci quand $k$ est imparfait, il suffit de passer à une extension radicielle finie convenable de $k$, ce qui ne change ni $H^2(X,1)$, ni $\Pic(X)\otimes \Q_l$, ni $G$.
\end{proof}

\begin{thm}\label{t1} $T(X)$ est vrai pour les $k$-variétés lisses de dimension $d$ si et seulement s'il est vrai pour les $k$-variétés projectives lisses de dimension $d$.
\end{thm}

\begin{proof}  Soit $X$ lisse de dimension $d$. Choisissons une immersion ouverte dense $X\inj X_0$ où $X_0$ est propre. D'après \cite[Th. 4.1]{dJ}, on peut trouver une altération $\pi:\tilde X\to X_0$ avec $\tilde X$ projective lisse et $\pi$ génériquement fini. Soit $U\subset X$ un ouvert tel que $\pi_{|\pi^{-1}(U)}$ soit fini et plat. 

Supposons $T(\tilde X)$ vrai. Par la proposition \ref{p2} b), $T(\pi^{-1}(U))$ est vrai. D'après le lemme \ref{l2}, $T(U)$ est donc vrai, et enfin $T(X)$ est vrai par la proposition \ref{p2} a).
\end{proof}

\section{Changement de corps de base}

\begin{prop}\label{t2} Soit $K/k$ une extension de corps de type fini. Alors $T(k)\iff T(K)$.
\end{prop}

\begin{proof} En quatre étapes; les deux premières et la dernière sont bien connues et valables en toute codimension; elles sont rappelées pour la clarté de l'exposition. Pour plus de précision, on note ici $G_K=Gal(K_s/K)$.

1) Soit $X$ lisse sur $K$,  connexe et de corps des constantes $L$. Comme $X$ est lisse, $L/K$ est séparable. Je dis que, avec des notations évidentes, $T(X/K)\iff T(X/L)$. En effet, le $G_K$-module $H^2(X/K,1)$ est induit du $G_L$-module $H^2(X/L,1)$, donc $H^2(X/K,1)^{G_K}\iso H^2(X/L,1)^{G_L}$.

2) L'énoncé est vrai si $K/k$ est finie séparable. En effet, $\Rightarrow$ résulte immédiatement de 1). Pour $\Leftarrow$, on se ramène à $K/k$ galoisienne en considérant sa clôture galoisienne; si $X$ est lisse sur $k$, $T(X_K)$ implique alors $T(X)$ en prenant les invariants sous $Gal(K/k)$.

3) Soit $k_0$ le sous-corps premier de $K$; montrons que $T(k_0)\Rightarrow T(K)$. Il suffit grâce au théorème \ref{t1}  de montrer que $T(k_0)$ implique $T(X)$ pour toute $K$-variété projective lisse $X$. L'argument est une version simplifiée de celle de \cite[th. 8.32 a)]{sheaf}.

On peut supposer $X$ connexe. Soit $L$ son corps des constantes, et soit $k_1$ la fermeture algébrique de $k_0$ dans $L$. Puisque $k_1$ est parfait, l'extension $L/k_1$ est régulière; choisissons-en un $k_1$-modèle lisse $S$. Quitte à remplacer $S$ par un ouvert, étendons $X$ en un $S$-schéma projectif lisse $f:\sX\to S$. Notant $\bar S=S\otimes_{k_1} k_s$, on a la suite spectrale de Leray (de $\Q_l[[G_{k_1}]]$-modules)
\[E_2^{p,q}=H^p(\bar S,R^qf_* \Q_l(1))\Rightarrow H^{p+q}(\sX,1).\]

D'après \cite{deligne1} (voir aussi \cite{deligne2}), le choix d'une section hyperplane lisse $\sY/S$ de $\sX/S$ et le théorème de Lefschetz difficile \cite[th. 4.1.1]{weilII} font dégénérer cette suite spectrale, montrant aussi que la filtration sur l'aboutissement est scindée\footnote{Le résultat précis de \cite[Prop. 2.4]{deligne1} ou de \cite[\S 2 ou \S 3]{deligne2} est que $Rf_*\Q_l$ est isomorphe à $\bigoplus_{i\ge 0} R^if_*\Q_l[-i]$ dans la catégorie dérivée.}. En particulier, l'homomorphisme ``edge'' $H^2(\sX,1)\to E_2^{0,2}=H^2(X,1)^{\pi_1(\bar S)}$ admet une section $G_{k_1}$-équivariante; par conséquent, $H^2(\sX,1)^{G_{k_1}}\to H^2(X,1)^{G_L}$ est \emph{surjectif}; en effet, $G_L\to \pi_1(S)$ est surjectif puisque $S$ est géométriquement connexe. Avec les notations de 1), on a donc $T(\sX/k_0)\Rightarrow T(\sX/k_1) \Rightarrow T(X/L)\Rightarrow T(X/K)$. (Noter que $k_1/k_0$ est séparable puisque $k_0$ est parfait.)

4) Finalement, montrons que $T(K)\Rightarrow T(k_0)$, ce qui terminera la démonstration. Soit, comme ci-dessus, $k_1$ la fermeture algébrique de $k_0$ dans $K$. Donnons-nous une $k_0$-variété projective lisse $X$; rappelons que $\NS(X_{k_1})\otimes \Q_l\to \NS(X_K)\otimes \Q_l$ est bijectif (c'est vrai en général pour les classes de cycles modulo l'équivalence algébrique, voir par exemple \cite[prop. 5.5]{adjoints} et sa preuve). Ceci montre que $T(X_K)\Rightarrow T(X_{k_1})$; mais d'autre part $T(X_{k_1})\Rightarrow T(X)$ par 2).
\end{proof}

\begin{thm}\label{t0} Soit $k_0$ le sous-corps premier de $k$. Alors $T(S)$ pour toutes les surfaces projectives lisses $S$ sur $k_0$ $\Rightarrow$ $T(k)$.
\end{thm}

\begin{proof} D'après la proposition \ref{t2}, on se ramène à $k=k_0$. Si $k=\Q$, soit $X$ une variété projective lisse connexe de dimension $d\ge 2$, de corps des constantes $k_1$. D'après le point 1) de la preuve de la rpoposition \ref{t2}, on peut remplacer $\Q$ par $k_1$. Choisissons un plongement complexe $k_1\inj \C$.  Par les théorèmes de comparaison, l'équivalence homologique $l$-adique pour $X\otimes_{k_1}\bar \Q$ coïncide avec la même pour $X\otimes_{k_1}\C$, qui coïncide avec l'équivalence homologique pour la cohomologie de Betti; notons $A^i_\hom(\bar X)$ les quotients correspondants.  Choisissons un $k_1$-plongement projectif $X\inj \P^N$, d'où un faisceau très ample $L$; le théorème de Lefschetz fort (pour la cohomologie de Betti) implique que 
$\cup c_1(L)^{d-2}:A^1_\hom(\bar X)\to A^{d-1}_\hom(\bar X)$ est injectif, et même \emph{bijectif} grâce au  théorème (1,1) de Lefschetz \cite[preuve du cor. 1]{lieb}. Comme $A^*_\hom(X)\iso A^*_\hom(\bar X)^{Gal(\bar \Q/k_1)}$, on a aussi un isomorphisme $\cup c_1(L)^{d-2}:A^1_\hom(X)\iso A^{d-1}_\hom(X)$. Mais, si $i:S\inj X$ est une surface (lisse, connexe) ample donnée par le théorème de Bertini, cet isomorphisme se factorise en
\[A^1_\hom(X)\by{i^*} A^1_\hom(S)\by{i_*} A^{d-1}_\hom(X)\]
et de même pour l'isomorphisme correspondant $H^2(X,1)\iso H^{2d-2}(X,d-1)$, de manière compatible aux classes de cycles. Une petite chasse aux diagrammes montre alors que $T(S)\Rightarrow T(X)$.

Si $k=\F_p$, Morrow se ramène d'abord au cas $\dim X\le 3$ par le théorème de Lefschetz faible pour la cohomologie $l$-adique \cite[cor. I.9.4]{fk} et pour le groupe de Picard \cite[cor. 4.9 b)]{SGA2}, puis au cas d'une surface dans \cite[th. 4.3]{morrow}. Le premier point est un peu délicat, comme me l'a fait remarquer Juan Felipe Castro Cárdenas: il n'est pas clair que, pour le groupe de Picard, Lefschetz faible soit vrai pour les diviseurs \emph{réduits}, en l'absence du théorème d'annulation de Kodaira (cf. \cite[rem. 4.10]{SGA2}). Néanmoins, l'argument de la preuve de \cite[th. 4.3]{morrow} pour réduire le cas de la dimension $3$ à celui de la dimension $2$ marche aussi bien, et même mieux, pour réduire le cas de la dimension $d+1$ à celui de la dimension $d$ quand $d\ge 3$: dans ce cas, toutes les inclusions horizontales du diagramme de la page 3495 sont des égalités.
\end{proof}

\begin{rque}\label{r2} D'après \cite{lrs}, la conjecture $T(S)$ pour les surfaces $S$ sur un corps fini $k$ est équivalente à la conjecture de Birch et Swinnerton-Dyer pour les jacobiennes de courbes sur les corps de fonctions d'une variable $K/k$. Cette dernière implique la même conjecture pour toute variété abélienne $A$ définie sur $K$: en effet, ladite conjecture est équivalente à la finitude de la composante $l$-primaire $\Sha(K,A)\{l\}$ du groupe de Tate-\v Safarevi\v c $\Sha(K,A)$ \cite{schneider,kt}. Si $i:C\inj A$ est une courbe ample, l'homomorphisme $i_*:J(C)\to A$ est surjectif (Weil, voir \cite[Lemma 2.3]{murre}), donc il existe $\sigma:A\to J(C)$ tel que $i_*\circ  \sigma$ soit la multiplication par un entier $n>0$, d'où $n\Sha(K,A)\subset i_* \Sha(K, J(C))$. Mais on sait que $\Sha(K,A)\{l\}$ est de cotype fini, donc la finitude de $\Sha(K, J(C))\{l\}$ implique celle de $\Sha(K,A)\{l\}$.
\end{rque}

À ce stade, il est obligatoire de terminer avec la question évidente:

\begin{qn} Peut-on réduire le cas de caractéristique zéro à celui de la caractéristique positive?
\end{qn}

Par changement de base propre et lisse et par le théorème de \v Cebotarev, cette question est équivalente à la suivante:

\begin{qn} Soit $S$ une $\Q$-surface projective lisse, et soit $\alpha\in H^2(S,1)$. Supposons que, pour (presque) tout nombre premier $p$ de bonne réduction, la spécialisation de $\alpha$ en $p$ soit algébrique. Est-ce que $\alpha$ est algébrique?
\end{qn}


\begin{thebibliography}{II}
\bibitem{ambrosi} E. Ambrosi {\it A note on the behaviour of the Tate conjecture under finitely generated field extensions}, Pure Appl. Math. Q. {\bf 14} (2018), 515--527.
\bibitem{andre} Y. André {\it Pour une théorie inconditonnelle des motifs}, Publ. Math. IHÉS {\bf 83} (1996), 5-49.
\bibitem{weilII} P. Deligne {\it La conjecture de Weil, II}, Publ. Math. IHÉS {\bf 52} (1980), 137--252.
\bibitem{deligne1} P. Deligne {\it Théorème de Lefschetz et critères de dégénérescence de suites spectrales}, Publ. Math. IHÉS {\bf 35} (1968), 259--278. 
\bibitem{deligne2} P. Deligne {\it Décompositions dans la catégorie dérivée}, {\it in} Motives, Proc. Symposia Pure Math. {\bf 55} (1), AMS, 1994, 115--128.
\bibitem{fk} E. Freitag, R. Kiehl, Étale cohomology and the Weil conjecture, Springer, 1988. 
\bibitem{fulton} W. Fulton Intersection theory, Springer, 2ème éd., Springer, 1998.
\bibitem{SGA2} A. Grothendieck {\it Application aux schémas algébriques projectifs}, Exp. XII de SGA 2, nouvelle édition, Doc. mathématiques {\bf 4}, SMF, 2005, 109--134.
\bibitem{cycle} A. Grothendieck {\it La classe de cohomologie associée à un cycle} (rédigé par P. Deligne), {\it in}  SGA 4 1/2, Lect. Notes in Math. {\bf 569}, Springer, 1977, 129--153.
\bibitem{jannsen} U. Jannsen Mixed motives and algebraic $K$-theory, Lect. Notes {\bf 1400}, Springer, 1990.
\bibitem{dJ} A.J. de Jong {\it Smoothness, semi-stability and alterations}, Publ. Math. IHES {\bf 83} (1996), 51--93.
\bibitem{sheaf} B. Kahn {\it A sheaf-theoretic reformulation of the Tate conjecture}, prépublication, 1998, \url{https://arxiv.org/abs/math/9801017}.
\bibitem{glr} B. Kahn {\it The Geisser-Levine method revisited and algebraic cycles over a finite field}, Math. Ann. {\bf 324} (2002), 581--617.
\bibitem{kmp} B. Kahn,  J.P. Murre, C. Pedrini {\it On the transcendental part of the motive of a surface},  {\it in} Algebraic cycles and motives (for J.P. Murre's 75th birthday), LMS Lect. Notes Series {\bf 344} (2), Cambridge University Press, 2007, 143--202.
\bibitem{adjoints} B. Kahn {\it Motifs et adjoints}, Rendiconti Sem. Mat. Univ. Padova {\bf 139} (2018), 77--128.
\bibitem{kt} K. Kato, F. Trihan {\it On the conjectures of Birch and Swinnerton-Dyer in characteristic $p > 0$}, Invent. Math. {\bf 53} (2003), 537--592.
\bibitem{li-zhang} C. Li, W. Zhang {\it A note on Tate’s conjectures for abelian varieties}, Ess. Numb. Th. {\bf 1} (2022), 41--50.
\bibitem{lrs} S. Lichtenbaum ; N. Ramachandran ; T. Suzuki {\it The conjectures of Artin-Tate and Birch-Swinnerton-Dyer}, EPIGA {\bf 6} (2022), \url{https://arxiv.org/abs/2101.10222}. 
\bibitem{lieb} D. Lieberman {\it Numerical and homological equivalence of algebraic cycles on Hodge manifolds}, Amer. J. Math. {\bf 90} (1968), 366--374.
\bibitem{matsusaka} T. Matsusaka {\it The criteria for algebraic equivalence and the torsion group}, Amer. J. Math. {\bf 79} (1957), 53--66.
\bibitem{morrow} M. Morrow {\it A variational Tate conjecture in crystalline cohomology}, J. Eur. Math. Soc. {\bf 21} (2019),  3467--3511.
\bibitem{murre} J. P. Murre: On the motive of an algebraic surface,  J. Reine Angew. Math. {\bf 409} (1990), 190--204.
\bibitem{schneider} P. Schneider {\it Zur Vermutung von Birch und Swinnerton-Dyer über globalen Funktionenkörpern}, Math. Ann. {\bf 260} (1982), 495--510.
\bibitem{tate0} J. Tate {\it Algebraic cycles and poles of zeta functions}, {\it in}  Arithmetical Algebraic Geometry (Conference, Purdue University, Lafayette, Ind., 1963). Sous la dir. d’O.F.G. Schilling. Harper \& Row, 1965, 93--110.
\bibitem{analog} J. Tate {\it On the conjectures of Birch and Swinnerton-Dyer and a geometric analog},  Sém. Bourbaki, années 1964/65 --1965/66, exposé {\bf 306}.
\bibitem{tate} J. Tate {\it Conjectures on algebraic cycles in $l$-adic cohomology}, {\it in} Motives, Proc. Symposia
Pure Math. {\bf 55} (1), AMS, 1994, 71--83.
\bibitem{weil} A. Weil Variétés abéliennes et courbes algébriques, Hermann, 1948.
\end{thebibliography}
\end{document}